\documentclass[reqno, 12pt]{amsart}
\usepackage[body={7in,9.5in},top=1in, left=0.8in ]{geometry}
\usepackage{enumerate}
\usepackage{enumitem}
\usepackage{cancel}
\usepackage{enumitem}
\usepackage{amssymb}
\usepackage{microtype}
\usepackage{comment}

\newtheorem{theorem}{Theorem}[section]

\newtheorem{lem}[theorem]{Lemma}

\newcommand{\N}{\mathbb N}

\newcommand{\C}{\mathbb C}

\newcommand{\R}{\mathbb R}

\newcommand{\K}{\mathbb K}
\newcommand{\Q}{\mathbb Q}

\newcommand{\ov}[1]{\bar{#1}}
\newcommand{\ord}{\textup{ord}}

\usepackage{lipsum}
\newcommand\blfootnote[1]{
    \begingroup
    \renewcommand\thefootnote{}\footnote{#1}
    \addtocounter{footnote}{-1}
    \endgroup
}

\author[D.Beck]{Dominik Beck$^*$}

\address[D.Beck]{Faculty of Mathematics and Physics\\
  Charles University\\
 Prague\\
 Czechia}
	
\email[D.Beck]{beckd@karlin.mff.cuni.cz}

\author[P.Ma\'ckowiak]{Piotr Ma\'ckowiak}

\address[P.Ma\'ckowiak]{Department of Nonlinear Analysis and Applied Topology\\
  Faculty of Mathematics and Computer Science\\
  Adam Mickiewicz University\\
  Uniwersytetu Pozna\'nskiego 4\\
  61-614 Pozna\'n\\
  Poland}
	
\email[P.Ma\'ckowiak]{piotr.mackowiak@amu.edu.pl}

\begin{document}
\title{Non-external Proofs of Lagrange Inversion Formula}
\begin{abstract}
The goal of the paper is to present two simple proofs of the Lagrange Inversion Formula for formal power series. Both proofs are non-external in the sense that they use concepts that do not go beyond the scope of basic tools of formal power series analysis. 
\end{abstract}
\maketitle
\blfootnote{$^*$Dominik Beck was supported by the Charles University, project GA UK No. 71224 and by Charles University Research Centre program No. UNCE/24/SCI/022.}

\section{Introduction} 
In a recent paper of Surya and Warnke \cite{surya2023lagrange}, a simple inductive proof of the Lagrange Inversion Formula (LIF) is presented. Our paper offers additional two - as we believe - simple proofs of LIF. In Subsection \ref{proof1} we present the first proof which is motivated by the inductive approach of Surya and Warnke. The proof we give is also by induction but the inductive steps are based on interplay two equivalent forms of LIF (see Subsect. \ref{proof1} for details).

Subsection \ref{proof2} contains the second proof which is partially motivated by but different from the algebraic proof of LIF presented in \cite[First Proof of Theorem 5.4.2, p.38]{Stanley}. The difference here is that we do not go beyond the calculus of formal powers series while the more general proof in Stanley's book \cite{Stanley} relies on the notion of Laurent formal series. 

As the aforementioned paper \cite{surya2023lagrange} contains a very nice introduction to the Lagrange Inversion Formula, basic definitions and properties of formal power series calculus, we refer the reader to that paper for applications, motivation and literature on LIF.  

We are going to prove the following theorem which is an equivalent formulation of Theorem 1 from \cite{surya2023lagrange}.
\begin{theorem}\label{thm:function}[functional form]
Let $\phi(x)=\sum_{n=0}^\infty\phi_nx^n$ and $f(x)=\sum_{n=0}^\infty f_nx^n$ be formal power series such that $\phi_0\neq 0$, $f_0=0$, $f_1\neq 0$ and $\ov{f}(x) = x \phi(\ov{f}(x))$, where $\ov{f}(x)$ is the compositional inverse of $f(x)$. For any formal power series $g(x)=\sum_{n=0}^\infty g_nx^n$ and any positive integer $n$,
\begin{equation}\label{LIF_Functional}
    [x^n] g(\ov{f}(x)) = \tfrac{1}{n} [x^{n-1}] g'(x) (x/f(x))^n = \tfrac{1}{n} [x^{n-1}] g'(x) \phi^n(x).
\end{equation}
\end{theorem}
Note that even although division of power series is ill-defined, we still use $x/f(x)$ as the abbreviation for the unique formal power series solution $y(x)$ of $y(x) f(x) = x$. The above statement of LIF with $g(x) = x^l$, $l \geq 0$, implies the next version of LIF we are going to prove as well.
\begin{theorem}\label{thm:SJ}[Schur-Jabotinsky form]
Let $\phi(x)=\sum_{n=0}^\infty\phi_nx^n$ and $f(x)=\sum_{n=0}^\infty f_nx^n$ be formal power series such that $\phi_0\neq 0$, $f_0=0$, $f_1\neq 0$ and $\ov{f}(x) = x \phi(\ov{f}(x))$, where $\ov{f}(x)$ is the compositional inverse of $f(x)$. For any formal power series $g(x)=\sum_{n=0}^\infty g_nx^n$ and any integers $n \geq 1$, $l \geq 0$ such that $n\geq l$,
\begin{equation}\label{LIF_SJ}
[x^n]\ov{f}(x)^l=\tfrac{l}{n}[x^{n-l}](x/f(x))^{n} = \tfrac{l}{n} [x^{n-l}] \phi^n(x).
\end{equation}
\end{theorem}
Theorem \ref{thm:SJ} implies Theorem \ref{thm:function}, since multiplying Equation (\ref{LIF_SJ}) by $g_l$ and summing over $l\geq 0$, we easily recover the original Equation \eqref{LIF_Functional} for any formal power series $g$ and any fixed positive integer $n$.

Let us notice that although we decided to explicitly assume $\phi_0\neq 0$, $f_0=0$, $f_1\neq 0$ it turns out that this is somewhat excessive (see \cite{surya2023lagrange} for a discussion on this).

\section{Notation, basic facts and assumptions}
Let $\N$ denote the set of positive integers and set $\N_0:=\N\cup\{0\}$. For a formal power series $f(x)$, its $n$-th coefficient is $f_n$ or $[x^n]f(x)$, $n\in \N_0$.

The product of two formal power series $f(x)$ and $g(x)$ is the formal power series $h(x)$ such that $h_n=\sum_{i=0}^nf_ig_{n-i}$, $n\in \N_0$. We denote the product of $f(x)$ and $g(x)$ by $f(x)g(x)$ or $(fg)(x)$. In the paper we consider formal power series with coefficients in an integral domain $\K$ (e.g. $\K\in\{\Q, \R, \C \}$); this assumption guarantees that the product is associative and commutative and for any pair of formal power series $f(x)$ and $g(x)$ their product is zero formal power series $\theta$ (every coefficient of $\theta$ equals $0\in \K$) if and only if either $g(x)=\theta$ or $f(x)=\theta$. The order of a non-zero formal power series $f(x)$ is the number $\ord f:=\min\{n\in \N_0:\, f_n\neq 0\}$; $\ord \theta:=+\infty$. It is known that the order of the product $f(x)g(x)$ equals the sum of orders of $f$ and $g$.   

For any $l\in \N_0$, $x^l$ denotes a formal power series whose all coefficients are $0$, except the coefficient $l$ which is $1\in \K$. The product of $k$ copies of a formal power series $f(x)$ ($k$-th power of $f$) is denoted by $f^k(x)$, $k\in \N$; $f^0(x):=x^0$.  A formal power series $g(x)$ possesses multiplicative inverse, that is, there exists a formal power series $f(x)$ such that $f(x)g(x)=x^0$ if and only if $g_0\neq 0$. Any formal power series $f(x)$ with $f_0=0$ and $f_1\neq 0$ is called almost unit. Any almost unit formal power series $f(x)$ possesses its compositional inverse $\ov{f}(x)$ (which is also almost unit), that is, it holds $f(\ov{f}(x))=\ov{f}(f(x))=x$. Here by $g(f(x))$ or, equivalently, $(g\circ f)(x)$, we mean the composition of a formal power series $g(x)$ with a formal power series $f(x)$, where $f(x)$ is a nonunit, $f_0=0$, that is, a formal power series $(g\circ f)(x)$ which is defined by $[x^n](g\circ f)(x):=\sum_{i=0}^\infty g_i[x^n]f^i(x)=\sum_{i=0}^n g_i[x^n]f^i(x)$, $n\in \N_0$; in this case we also write $g(f(x))=\sum_{i=0}^\infty g_if^i(x)$. Note that in order the infinite summation to be well-defined in the context of formal power series, we require that every infinite series must contain only finite number of non-zero terms for every $n$ (sequence admitting addition \cite{Niven}). This makes the composition formula written above well-defined since only first $n$ terms contribute to $x^n$ coefficient.

The derivative of a formal power series $f(x)$ is $f'(x)$ (or, equivalently, $(f(x))'$) defined by $[x^n]f'(x):=(n+1)[x^{n+1}]f(x)$, $n\in \N_0$. In our proofs we shall use the following properties of formal power series calculus. Let $f(x)$ and $g(x)$ be formal power series. Then, for any nonunit formal power series $h(x)$, 
\begin{enumerate}
\item product rule: $(fg)'(x)=f'(x)g(x)+f(x)g'(x)$,
\item power rule: $(f^{k+1})'(x)=(k+1)f^{k}(x) f'(x)$, $k\in \N_0$,
\item chain rule: $(f\circ h)'(x)=(f'\circ h)(x)h'(x)$ ,
\item term-by-term differentiation: $[x^n](f\circ h)'(x)=[x^n](f_0(h^0)'(x)+f_1(h^1)'(x)+f_2(h^2)'(x)+\ldots)$, $n\in \N_0$,
\item right distributive law: $((fg)\circ h)(x)=(f\circ h)(x)(g\circ h)(x)$.

\end{enumerate}
The notions and properties of formal power series presented in this section can be found in the classic work \cite{Niven} or in \cite[Chapter 1]{Henrici} or in a recent book \cite{g21}. A more general approach to the right distribute law is presented in \cite{BGM}.

\section{Proofs of LIF}
\subsection{Inductive approach to LIF}\label{proof1} The following proof is an adaptation of Surya and Warnke's proof \cite{surya2023lagrange} and exploits the interrelationships between the functional and Schur-Jabotinsky forms of the Lagrange inversion formula. 

\begin{proof}[Proof of Theorem \ref{thm:function}]
We proceed by induction, keeping in mind that as soon we prove it for a certain $n$ for any $g(x)$ (functional), it is immediately valid for $g(x) = x^l$ for any $l$ (Schur-Jabotinsky form) and vice versa. First, we show that LIF holds for $n=1$. We use its functional form. The RHS of Equation \eqref{LIF_Functional} gives
\begin{equation*}
    [x^0] g'(x) \phi^n(x) = g'(0) \phi(0) = g_1 \phi_0,
\end{equation*}
while the left--hand side equals
\begin{equation*}
    [x^1] g(\ov{f}(x)) = [x^1] (g_0 + g_1 \ov{f}(x) + g_2 \ov{f}^2(x) + \cdots) = [x^2] (g_0 + g_1 x \phi(\ov{f}(x)) + g_2 x^2 \phi^2(\ov{f}(x)) + \cdots) = g_1 \phi_0.
\end{equation*}
Next, let's assume LIF holds for a fixed integer $n$ (both of its forms), we show it implies that it also holds for $n+1$ and all $l$ in its Schur-Jabotinsky form (from which follows the functional form). We successively transform the left--hand side of Equation \eqref{LIF_SJ} -- for $n+1$ (and any $l$ fixed), it equals
\begin{equation*}
    [x^{n+1}] \ov{f}^{l+1}(x).
\end{equation*}
Writing $\ov{f}^{l+1}(x)$ as $\ov{f}(x) \ov{f}^l(x)$ and replacing $\ov{f}(x)$ by $x \phi(\ov{f}(x))$, we get the left--hand side of Equation \eqref{LIF_SJ} expressed as
\begin{equation*}
[x^{n+1}] \ov{f}^l(x) x \phi(\ov{f}(x))
\end{equation*}
or, equivalently,
\begin{equation*}
[x^n] \ov{f}^l(x) \phi(\ov{f}(x)).
\end{equation*}
However, taking the $n$-th term can be done by LIF itself (it is assumed it holds for $n$). Setting $g(x) = x^l \phi(x)$ in Equation \eqref{LIF_Functional} yields
\begin{equation*}
    \tfrac{1}{n} [x^{n-1}] (x^l \phi(x))' \phi^n(x).
\end{equation*}
Writing out the derivative, we get
\begin{equation*}
\tfrac{1}{n} [x^{n-1}](l x^{l-1} \phi^{n+1}(x) + x^l \phi'(x) \phi^n(x)),
\end{equation*}
which, after expanding the coefficient extraction operator and grouping the second term under a single derivative, gives the following expression 
\begin{equation*}
\tfrac{1}{n} \left( l [x^{n-l}] \phi^{n+1}(x) + [x^{n-l-1}] \tfrac{1}{n+1} (\phi^{n+1}(x))' \right).
\end{equation*}
Finally, simplifying the second term via product rule, we get
\begin{equation*}
\tfrac{1}{n} \left( l [x^{n-l}]\phi^{n+1}(x) + \tfrac{n-l}{n+1} [x^{n-l}] \phi^{n+1}(x) \right)
\end{equation*}
which is equivalent to
\begin{equation*}
\tfrac{l+1}{n+1} [x^{n-l}] \phi^{n+1}(x),
\end{equation*}
but this is the right--hand side of Equation \eqref{LIF_SJ}. The proof is now finished.
\end{proof}

We briefly discuss the similarity of our proof with the one by Surya and Warnke \cite{surya2023lagrange}. In their proof, they select $g(x) = \phi(x)^l$, which results in a nested induction from $l$ to $n$ in the inductive step. However, taking $g(x) = x^l \phi(x)$ as above enabled us to use the induction directly from $n$ to $n + 1$.

\subsection{Calculus of formal power series approach to LIF}\label{proof2} The below proof uses the right distributive law, basic formulas of formal power series calculus and properties of order of a formal power series we introduced in section 2 of the paper. However, it also needs a bit technical (though simple) lemma. 
\begin{lem}\label{lm:1}
Let $g=\sum_{n=0}^\infty g_n x^n$ be a multiplicatively invertible formal power series, that is, $g_0\neq 0$. For every $j,s\in \N_0$, it holds 
$$[x^s] (x^jg^{-s+j}(x)+x^{j+1}g^{-s+(j-1)}(x)g'(x))=\delta_{j,s},$$
where $\delta_{j,s}$ denotes the Kronecker delta.
\end{lem}
\begin{proof}
Since the order of the formal power series $x^jg^{-s+j}(x)$ is $j$ and the order of $x^{j+1}g^{-s+(j-1)}(x)g'(x)$ is not less than $j+1$ (it may be $+\infty$), the claim is true for $j\geq s$.

Suppose that $j<s$. We have $[x^{s-j-1}](g^{-s+j})'(x)=(s-j)[x^{s-j}]g^{-s+j}(x)$ (by definition of the formal derivative) and, at the same time, $[x^{s-j-1}](g^{-s+j})'(x)=(-s+j)[x^{s-j-1}](g^{-s+j-1}(x)g'(x))$ (by power rule). Hence, $[x^{s-j}]g^{-s+j}(x)+[x^{s-j-1}](g^{-s+j-1}(x)g'(x))=0$ which is equivalent to $[x^s](x^jg^{-s+j}(x)+x^{j+1}g^{-s+j-1}(x)g'(x))=0$. The lemma is proved.
\end{proof}

\begin{proof}[Proof of Theorem \ref{thm:SJ}]
Direct calculations show that the claim is true for $n=l>0$ or $l=0$ and we may assume that $n>l>0$.\\
Let $\ov{f}(x)$ be the compositional inverse of $f(x)$. Then, by the right distributive law,  $(\ov{f}^l\circ f)(x)=(\ov{f}\circ f)^l(x)=x^l$. Hence, $(\ov{f}^l\circ f )'(x)=lx^{l-1}$. On the other hand, by definition of composition, we have $(\ov{f}^l\circ f)(x)=\ov{f}^l_{l}f^{l}(x)+\ov{f}^l_{l+1}f^{l+1}(x)+\ldots$, where $\ov{f}^l_{n}:=[x^n]\ov{f}^l (x)$, and, subsequently, $$(\ov{f}^l\circ f)'(x)=\sum_{i=l}^\infty i\ov{f}^{l}_i f^{i-1}(x)f'(x).$$
Thus, 
\begin{equation}\label{eq:1}\sum_{i=k}^\infty i\ov{f}^{l}_i f^{i-1}(x)f'(x)=lx^{l-1}.
\end{equation}
Now, let $\hat{f}(x)$ be a formal powers series defined by $[x^n]\hat{f}(x):=[x^{n+1}]f(x)$ for all $n\in \N_0$, that is, $\hat{f}(x)$ is obtained from $f$ by backshifting the coefficients of $f$. Since $f(x)=x\hat{f}(x)$ and $f'(x)=\hat{f}(x)+x\hat{f}'(x)$, formula (\ref{eq:1}) can be written as 
$$\sum_{i=l}^\infty i\ov{f}^{l}_ix^{i-1}(\hat{f}^{i}(x)+x\hat{f}^{i-1}(x)\hat{f}'(x))=lx^{l-1}.$$
But the left--hand side of the above equation is $$x^{l-1}\sum_{i=l}^\infty i\ov{f}^{l}_ix^{i-l}(\hat{f}^{i}(x)+x\hat{f}^{i-1}(x)\hat{f}'(x)),$$
so we get 
$$\sum_{i=l}^\infty i\ov{f}^{l}_ix^{i-l}(\hat{f}^{i}(x)+x\hat{f}^{i-1}(x)\hat{f}'(x))=lx^0.$$
Multiplying the last equality by $\hat{f}^{-n}(x)$ we arrive at 
$$\sum_{i=l}^\infty i\ov{f}^{l}_i(x^{i-l}\hat{f}^{i-n}(x)+x^{i-l+1}\hat{f}^{i-n-1}(x)\hat{f}'(x))=l\hat{f}^{-n}(x).$$
By Lemma \ref{lm:1} (with $g$ and $s$ replaced with $\hat{f}$ and $n-l$, respectively), for $j=0,1,\ldots,n-l-1$, we have $[x^{n-l}](x^{j}\hat{f}^{l-n+j}(x)+x^{j+1}\hat{f}^{l-n+j-1}(x)\hat{f}'(x))=0$. Moreover, $\ord(x^{n-l+1}\hat{f}^{-1}(x)\hat{f}'(x))=n-l+1$ and $\ord(x^{m-k}g(x))>n-l$ for $m>n$ and any formal power series $g$. Hence, 
$$[x^{n-l}]\sum_{i=l}^\infty i\ov{f}^{l}_i(x^{i-l}\hat{f}^{i-n}(x)+x^{i-l+1}\hat{f}^{i-n-1}(x)\hat{f}'(x))=[x^{n-l}]n\ov{f}^l_n x^{n-l}=n\ov{f}^l_n=n[x^n]\ov{f}^l(x).$$
Since $[x^{n-l}]l\hat{f}^{-n}(x)=l[x^{n-l}](x/f(x))^{n}$, the proof is completed.
\end{proof}

\end{document}